\documentclass[12pt]{article} 

\usepackage{amsmath}
\usepackage{amsthm}
\usepackage{amsfonts}
\usepackage{mathrsfs}
\usepackage{stmaryrd}
\usepackage{setspace}
\usepackage{fullpage}
\usepackage{amssymb}
\usepackage{breqn}
\usepackage{enumitem}
\usepackage{bbold} 
\usepackage{authblk}
\usepackage{comment}
\usepackage{hyperref}
\usepackage{pgf,tikz}
\usepackage{graphicx}
\usepackage{subcaption}

\newtheorem{thm}{Theorem}[section]

\newtheorem{lemma}[thm]{Lemma}
\newtheorem{conjecture}[thm]{Conjecture}
\newtheorem{proposition}[thm]{Proposition}

\newtheorem{cor}[thm]{Corollary}

\newtheorem{clm}[thm]{Claim}

\newcommand\ex{\ensuremath{\mathrm{ex}}}

\newcommand\cH{{\mathcal H}}

\newcommand\cN{{\mathcal N}}

\newcommand{\ignore}[1]{}

\title{The Turán number of Berge book hypergraphs}
\author{Dániel Gerbner}
\date{\small Alfr\'ed R\'enyi Institute of Mathematics}

\begin{document}

\maketitle

\begin{abstract}
Given a graph $G$, a Berge copy of $G$ is a hypergraph obtained by enlarging the edges arbitrarily. Gy\H ori in 2006 showed that for $r=3$ or $r=4$, an $r$-uniform $n$-vertex Berge triangle-free hypergraph has at most $\lfloor n^2/8(r-2)\rfloor$ hyperedges if $n$ is large enough, and this bound is sharp.

The book graph $B_t$ consists of $t$ triangles sharing an edge. Very recently, Ghosh, Győri, Nagy-György, Paulos, Xiao and Zamora showed that a 3--uniform $n$-vertex Berge $B_t$-free hypergraph has at most $n^2/8+o(n^2)$ hyperedges if $n$ is large enough. They conjectured that this bound can be improved to $\lfloor n^2/8\rfloor$. 

We prove this conjecture for $t=2$ and disprove it for $t>2$ by proving the sharp bound $\lfloor n^2/8\rfloor+(t-1)^2$. We also consider larger uniformity and determine the largest number of Berge $B_t$-free $r$-uniform hypergraphs besides an additive term $o(n^2)$. We obtain a similar bound if the Berge $t$-fan ($t$ triangles sharing a vertex) is forbidden.
\end{abstract}

\section{Introduction}

Given a graph $G$ and a hypergraph $\cH$, we say that $\cH$ is a Berge copy of $G$ (in short: a Berge-$G$) if $V(G)\subset V(\cH)$ and there is a bijection $f:E(G)\subset E(\cH)$ such that for each edge $e\in E(G)$ we have $e\subset f(e)$. This notion was introduced by Gerbner and Palmer \cite{gp1}, generalizing the well established notion of hypergraph cycles and paths due to Berge.

Extremal problems concerning Berge copies have attracted a lot of researchers. The most studied question is the largest number of edges in $r$-uniform $n$-vertex Berge-$F$-free hypergraphs, which we denote by $\ex_r(n,\textup{Berge-}F)$ (note that for $r=2$ this is the same as the ordinary Tur\'an number $\ex(n,F)$).
One of the earliest results on extremal problems for Berge hypergraphs is due to Gy\H ori \cite{gyori}.

\begin{thm}[Gy\H ori]\label{triangle} If $n\ge 100$, then
$\ex_3(n,\textup{Berge-}K_3)=\lfloor n^2/8\rfloor$.
\end{thm}

This was very recently extended to an asymptotic bound for book graphs by Ghosh, Győri, Nagy-György, Paulos, Xiao and Zamora \cite{ggnpxz}. The book graph $B_t$ consists of $t$ triangles sharing an edge, i.e. has vertices $u,v,w_1,\dots,w_t$ and edges $uv$, $uw_i$ and $vw_i$ for $i\le t$. We say that $u$ and $v$ are the rootlet vertices, $uv$ is the rootlet edge and $w_i$ are the page vertices of the book.

\begin{thm}[Ghosh, Győri, Nagy-György, Paulos, Xiao and Zamora \cite{ggnpxz}]\label{etal}

$\ex_3(n,\textup{Berge-}B_t)=\lfloor n^2/8\rfloor+o(n^2)$.
\end{thm}

They conjectured that the $o(n^2)$ term can be omitted for $n$ large enough. We prove their conjecture for $t=2$ and disprove it for larger $t$ by the following theorem, which determines the exact value of $\ex_3(n,\textup{Berge-}B_t)$ for $n$ large enough.

\begin{thm}\label{3uni}  If $n$ is large enough, then we have

\textbf{(i)} $\ex_3(n,\textup{Berge-}B_2)=\lfloor n^2/8\rfloor$.

\textbf{(ii)}
$\ex_3(n,\textup{Berge-}B_t)=\lfloor n^2/8\rfloor+(t-1)^2$. 
\end{thm}

We also consider larger uniformity. Gerbner \cite{gerb} proved that if $r\ge 2t+3$, then $\ex_r(n,\textup{Berge-}B_t)=o(n^2)$. For other values of $r$, we determine $\ex_r(n,\textup{Berge-}B_t)$ asymptotically.

\begin{thm}\label{runi} Let $k=\min \{r-1,t+1\}$. If $ r\le 2t+2$, then
$\ex_r(n,\textup{Berge-}B_t)=n^2/4k(r-k)+o(n^2)$. 
\end{thm}

Our methods also work for the $t$-fan $F_t$, which consists of $t$ triangles sharing a vertex.  Gerbner \cite{gerb} proved that if $r\ge 4t+2$, then $\ex_r(n,\textup{Berge-}F_t)=o(n^2)$.

\begin{thm}\label{kfan} Let $k=\min \{r-1,2t\}$. If $ r\le 4t+1$, then
$\ex_r(n,\textup{Berge-}F_t)=n^2/4k(r-k)+o(n^2)$. 
\end{thm}

The rest of this paper is organized as follows. In Section 2, we describe the tools we need, in particular the well-known connection to generalized Tur\'an problems. Section 3 contains the proofs of Theorems \ref{3uni}, \ref{runi} and \ref{kfan}. We finish the paper with some concluding remarks in Section 4.

\section{Preliminaries}

We use some fundamental results of graph theory. The Erd\H os-Stone-Simonovits theorem \cite{ES1966,ES1946} states that for any graph $F$ with chromatic number $p$ we have $\ex(n,F)=(1-\frac{1}{p-1})\binom{n}{2}+o(n^2)$. The lower bound is obtained by the Tur\'an graph $T(n,p-1)$, which is the complete $(p-1)$-partite graph with each part of order $\lfloor n/(p-1)\rfloor$ or $\lceil n/(p-1)\rceil$. A theorem of Simonovits \cite{sim2} states that if $F$ has an edge such that deleting that edge results in a graph with chromatic number $p-1$ and $n$ is large enough, then $\ex(n,F)=|E(T(n,p-1))|$, in particular $\ex(n,B_t)=\lfloor n^2/4\rfloor$. The Erd\H os-Simonovits stability theorem \cite{erd1,erd2,sim} states that for any $\varepsilon>0$ there exists a $\delta>0$ such that if an $F$-free $n$-vertex graph $G$ has at least $|E(T(n,p-1)|-\delta n^2$ edges, then $G$ can be obtained from $T(n,p)$ by adding and deleting at most $\varepsilon n^2$ edges. The removal lemma \cite{efr} states that if an $n$-vertex graph has $o(n^{|V(H)|})$ edges, then it can be made $H$-free by deleting $o(n^2)$ edges. 

\smallskip

Given graphs $H$, $G$ and $F$, we denote by $\cN(H,G)$ the number of copies of $H$ in $G$. The generalized Tur\'an number $\ex(n,H,F):=\max \{\cN(H,G): \text{ $G$ is an $n$-vertex $F$-free graph}\}$. After several sporadic result, the systematic study of this quantity was initiated by Alon and Shikhelman \cite{as}.

The connection of generalized Tur\'an numbers and Berge hypergraphs was established by Gerbner and Palmer \cite{gp2} in the following statement.

\begin{lemma}\label{celeb} For any integers $n$ and $r$ and any graph $F$, we have
$\ex(n,K_r,F)\le \ex_r(n,\textup{Berge-}F)\le \ex(n,K_r,F)+\ex(n,F)$.
\end{lemma}

An improvement was obtained in \cite{fkl} and \cite{gmp} independently.

\begin{lemma}\label{celeb2} For any integers $n$ and $r$ and any graph $F$, there is an $F$-free $n$-vertex graph $G$ such that each of its edges is colored either blue or red and $\ex_r(n,\textup{Berge-}F)$ is at most the number of blue copies of $K_r$ plus the number of the red edges in $G$.
\end{lemma}

The \textit{shadow graph} $G$ of a hypergraph $\cH$ is the graph with the same vertex set such that $uv$ is an edge of $G$ if and only if there is a hyperedge of $\cH$ that contains both $u$ and $v$. If $\cH$ is $r$-uniform, then each hyperedge creates a $K_r$, thus $\cH$ has at most $\cN(K_r,G)$ hyperedges. Note that $G$ can contain $F$ even if $\cH$ is Berge-$F$-free.

We will use the following results on generalized Tur\'an numbers.

\begin{proposition}[Alon, Shikhelman \cite{as}]\label{alsh}
For any $t$ we have $\ex(n,K_3,B_t)=o(n^2)$.
\end{proposition}

We will need the following simple corollary.

\begin{cor}\label{alshcor}
If $r\ge 3$, then $\ex(n,K_r,B_t)=o(n^2)$. 
\end{cor}

\begin{proof}
Let $G$ be an $n$-vertex $B_t$-free graph. We can upper bound the number of copies of $K_r$ the following way. We pick a triangle $uvw$ first, $o(n^2)$ ways, and then we pick $r-3$ common neighbors of $u$ and $v$. We count every copy of $K_r$ $\binom{r}{3}$ many times, and we might count some $r$-vertex subgraphs that are not cliques. As $u$ and $v$ have at most $t-1$ common neighbors, we have at most $\binom{t-2}{r-3}$ ways to pick the $r-3$ additional vertices, thus $\cN(K_r,G)\le \binom{t-2}{r-3}\cN(K_3,G)/\binom{r}{3}$.
\end{proof}

\begin{proposition}[Alon, Shikhelman \cite{as}]\label{alsh2}
For any $t$ we have $\ex(n,K_3,F_t)=O(n)$.
\end{proposition}

\begin{proposition}\label{genturfan}
For any $t$ and $r\ge 3$ we have $\ex(n,K_r,F_t)=O(n)$.
\end{proposition}

The main idea of our proof comes from the proof of the reduction lemma (Lemma 1.6) in \cite{gyolem}.

\begin{proof} We use induction on $r$, the base case $r=3$ is Proposition \ref{alsh2}. Assume now that $r>3$ and the statement holds for $r-1$.
Let $G$ be an $F_t$-free graph. We go through the copies of $K_r$ in $G$ one by one, and always mark a $K_{r-1}$ inside the $K_r$. Out of the $r$ possible copies of $K_{r-1}$, we mark one that was marked the least times earlier. In the case there are several such $(r-1)$-cliques (e.g. for the first $K_r$), we choose arbitrarily. At each step in the process, we call the number of times a copy $Q$ of $K_{r-1}$ had been marked the \textit{multiplicity} of $Q$. At the end of this process, we denote by $m(Q)$ the multiplicity of $Q$. If $m(Q)\le 4t^2$ for every $(r-1)$-clique, then $\cN(K_r,G)\le 4t^2\cN(K_{r-1},G)$ and we are done by the induction hypothesis.

Assume that $m(Q)>4t^2$ and $Q$ has vertices $u_1,\dots,u_{r-1}$. Let us describe a simplified version of the proof to make it easier to understand. Then we will point out the problem with it, and show how to fix that problem.

The last time $Q$ was marked, every $(r-1)$-clique in a $K_r$ containing $Q$ had multiplicity at least $4t^2$. Let $v_1$ be the last vertex of the $K_r$, and $Q_1$ be the $(r-1)$-clique we obtain from $Q$ by adding $v_1$ and deleting $u_1$. Then $m(Q_1)\ge 4t^2$. When $Q_1$ reached multiplicity $4t^2$, again it happened because of a $K_r$, thus there is a vertex $v_2$ such that replacing $u_2$ with $v_2$ we obtain an $(r-1)$-clique with multiplicity at least $4t^2-1$. 
We continue this way, but we need to describe what vertex is replaced by a new one. We always keep $u_3,\dots,u_{r-1}$, and each time we replace $v_{i-2}$ with $v_i$. This means that each time
we find a vertex $v_i$ such that $u_3,\dots,u_{r-1},v_{i-1},v_i$ induce an $(r-1)$-clique. 

However, $v_i$ could be one of the vertices $v_j$ with $j<i$. It is easy to avoid this though. When we pick $v_i$ with $i\le 2t$, we use that the $(r-1)$-clique $Q_{i-1}$ on $u_3,\dots,u_{r-1},v_{i-2},v_{i-1}$ has  multiplicity at least $p$ for some $p$. Instead of using the $r$-clique where $Q_{i-1}$ reached multiplicity $p$, we use one of the $r$-cliques where $Q_{i-1}$ reached multiplicity at least $p-2t$. They each contain a vertex different from $u_3,\dots,u_{r-1},v_{i-2},v_{i-1}$, and from each other, thus at most $i-3$ can be equal to $v_j$ with $j<i-1$. Hence we can pick a $v_i$ different from all the earlier $v_j$ and from $u_3$. We can this way pick vertices up to $v_{2t}$, since after $v_{2t-1}$ we have multiplicity at least $2t$.

Clearly, there is an $F_t$ in $G$ with vertices $u_3,v_1, \dots, v_{2t}$ and edges $u_3v_1,\dots, u_3v_{2t},v_1v_2,\dots,v_{2t-1}v_{2t}$, a contradiction.
\end{proof}

We will heavily use the fact that the above bounds do not depend on $t$. For this reason, it is going to be enough for us to forbid larger books or fans in certain subgraphs of the shadow graphs. The next proposition provides the tool, and also hints why our methods give sharp bounds only for books and fans.

\begin{proposition}\label{largebook}
\textbf{(i)} If $\cH$ is Berge-$B_t$-free, then there is no $B_{3rt}$ in the shadow graph such that each of the $3rt$ triangles is the core of a Berge triangle.

\textbf{(ii)} If $\cH$ is Berge-$F_t$-free, then there is no $F_{3rt}$ in the shadow graph such that each of the $3rt$ triangles is the core of a Berge triangle.
\end{proposition}

\begin{proof}
Assume indirectly that there is an $F_{3rt}$ in the shadow graph and each of the triangles is a core of the Berge triangle. Let $u,v_1,\dots,v_{3rt}, w_1,\dots,w_{3rt}$ be its vertices and $uv_i$, $uw_i$, $v_iw_i$ be its edges for $i\le 3rt$. We denote by $T_i$ the triangle $uv_iw_i$. We take three distinct hyperedges $h_1,h_2,h_3$ of $\cH$ such that $h_1$ contains $uv_1$, $h_2$ contains $uw_1$ and $h_3$ contains $v_1w_1$ (they exist since the triangle $uv_1w_1$ is the core of a Berge triangle). These three hyperedges contain at most $3r-6$ other vertices, thus without loss of generality they avoid $v_2,\dots,v_{3rt-3r+6}$ and $w_2,\dots,w_{3rt-3r+6}$. Then for $2\le i\le 3rt-3r+6$, we have that the hyperedges containing $uv_i$, $uw_i$ or $v_iw_i$ are distinct from $h_1,h_2,h_3$. Then we pick 3 distinct hyperedges containing $uv_2$, $uw_2$ and $v_2w_2$. 

In general, at Step $j$ for $j<t$, we have already picked $3j-3$ distinct hyperedges containing the $3j-3$ earlier edges $uv_i$, $uw_i$, $v_iw_i$ for $i<j$, and those hyperedges do not contain any of $v_j,\dots, v_{3rt-(j-1)(3r-6)},w_j,\dots, w_{3rt-(j-1)(3r-6)}$. Then we pick three distinct hyperedges for $uv_j$, $uw_j$ and $vw_j$. Afterwards, we relabel the triangles $T_{j+1},\dots, T_{3rt-(j-1)(3r-6)}$ such that the triangles that share vertices with the three new hyperedges come last, and relabel the  vertices such that $v_i$ and $w_i$ are in $T_i$ according to the new labels.  

Observe that we can execute these steps as long as we have a triangle remaining, i.e. whenever $(j-1)(3r-6)<3rt$. This means we can execute Step $t$, and obtain a Berge-$F_t$, a contradiction completing the proof of \textbf{(ii)}.

Let us continue with the sketch of the proof of \textbf{(i)}.
We proceed as in the case of fans. The only difference is that at each step, we only take two new hyperedges instead of three, thus we omit the details.
\end{proof}

Lu and Wang \cite{lw} initiated the study of the cover Tur\'an number, which is the largest number of edges in the shadow graph of a Berge-$F$-free hypergraph. Among other results, they proved the following.

\begin{thm}[Lu and Wang \cite{lw}]\label{luwa}
If $F$ has chromatic number $p$, then the number of edges in the shadow graph of a Berge-$F$-free $n$-vertex hypergraph is at most $(1-\frac{1}{p-1})\binom{n}{2}+o(n^2)$.
\end{thm}

\section{Proofs}

Let us start with the constructions. We define several very similar hypergraphs. We start with the basic \textbf{Construction 1}, giving the lower bound in Theorems \ref{triangle} and \ref{etal}. Note that we introduce here more notation than needed, as we will use the notation for other constructions. We take four sets of size $\lfloor n/4\rfloor$. We add $n-4\lfloor n/4\rfloor$ elements to them. If we add one element, we add it to the last set, if we add two elements, we add them to the last two sets, and if we add three elements, we add them to the first three sets. We denote the resulting sets by $A_1,A_2,A_3,A_4$ in this order. We let $A_i=\{a_1^{i},\dots, a_{|A_i|}^{i}\}$.
We take the hyperedges of the form $\{a_1^{i}, a_2^{i},a_3^j\}$ and of the form $\{a_1^{i}, a_2^{i},a_4^j\}$ for every $i$ and $j$. It is well-known and easy to see that this hypergraph contains no Berge triangle and has $\lfloor n^2/8\rfloor$ hyperedges. We say that $A_1\cup A_2$ is the \textit{left part} and $A_3\cup A_4$ is the \textit{right part}, and we say that $a_i^{1}$ and $a_i^{2}$ are \textit{twins}, just like $a_i^{3}$ and $a_i^{4}$. So the hyperedges are obtained by taking a twin from the left part and a vertex from the right part. In the case $n$ is odd, there is a vertex without a twin, which we call \textit{extra} element.

We slightly modify the above construction to obtain \textbf{Construction 2}. We take the hyperedges of the form $\{a_1^{i}, a_2^{i},a_3^j\}$ and of the form $\{a_1^{i}, a_2^{i},a_4^j\}$ for $i\le t-1$ and every $j$, and the hyperedges of the form $\{a_3^{i},a_4^{i},a_1^{j}\}$ or $\{a_3^{i},a_4^{i},a_2^{j}\}$ for $i\le t-1$ and $j>t-1$. In other words, we take one of the first $t-1$ twins from the left part and another vertex from the right part, or we take one of the first $t-1$ twins from the right part and another vertex from the left part that is not among the first $t-1$ twins. Observe that similarly to Construction 1, if we take a twin from the left part and a twin from the right part, we have exactly two hyperedges on those four vertices. Furthermore, in both constructions each hyperedge contains a twin and a vertex from the other side, thus we have the same number $\lfloor n^2/8\rfloor$ hyperedges if $n$ is even. If $n$ is odd, then the extra element is in $|A_1|$ hyperedges both in Constructions 1 and 2, thus again we have $\lfloor n^2/8\rfloor$ hyperedges.

We obtain \textbf{Construction 3} from Construction 2 by adding the hyperedges of the form $\{a_1^{i},a_3^j,a_4^j\}$ for every $i,j\le t-1$. In other words, we take one of the first $t-1$ twins in both parts, and add one of the missing hyperedges. Clearly, Construction 3 has $\lfloor n^2/8\rfloor+(t-1)^2$ hyperedges.

We claim that if $t>2$, then Construction 3 does not contain a Berge-$B_t$. Indeed, assume first that the rootlet vertices $u,v$ of a Berge-$B_t$ are in the same part. Then they have to be twins, and their neighbors are in the other part. If a vertex in the other part does not belong to the first $t-1$ twins, then it are connected to $u$ and $v$ by a single hyperedge, a contradiction. For the first $t-1$ twins, each pair of twin vertices is connected to $u$ and $v$ by three hyperedges, thus only one of them can be a page vertex. If $u$ and $v$ are in different parts, then they have only two common neighbors: the twins of $u$ and $v$.

We continue with the generalization to uniformity $r$. It is more convenient to define the sets of twins in this case. We again place roughly $n/2$ elements to the left side and to the right side, partitioned into sets of twins. More precisely, we take $ m=\lfloor\frac{\lceil n/2\rceil}{k}\rfloor$ sets $B_1,\dots,B_m$, each of size $k$, and $m'=\lfloor \frac{n-km}{r-k}\rfloor$ sets $C_1,\dots,C_{m'}$, each of size $r-k$. \textbf{Construction 4} is defined the following way. We take the $r$-edges of the form $B_i\cup C_j$. We clearly have $n^2/4k(r-k)-O(1)$ such hyperedges.

We claim that Construction 4 (with $k=\min\{r-1,t+1\})$ does not contain a Berge-$B_t$. Assume first that the rootlet vertices $u,v$ of a Berge-$B_t$ are in the same part. Then they are in the same $B_i$ or $C_i$, and every vertex in the other part is connected to them only by a single hyperedge, thus cannot be connected to both $u$ and $v$ in the core. Vertices in the same part are connected to $u$ and $v$ only if they are in the same $B_i$ or $C_i$, thus there are at most $k-2\le t-1$ page vertices. Assume now that $u$ and $v$ are in different parts, $u\in B_i$, $v\in C_j$. Then their common neighbors could only be in $B_i$ or $C_j$. However, there is a single hyperedge that contains $u$ and $v$, and it is the only hyperedge that would connect $u$ to the elements of $C_j$ or $v$ to the elements of $B_i$.



We remark that in the case $k$ is divisible by $r-k$, we can obtain some improvement similarly to the 3-uniform case. 

We also claim that Construction 4 (with $k=\min\{r-1,2t\})$ does not contain a Berge-$F_t$.  
Indeed, let $u$ be the vertex connected to every other vertex in the $t$-fan. Then the other vertices of the $t$-fan are either twins of $u$ or are in the other part. If a vertex $v$ of the $t$-fan is in the other part, then the third vertex $w$ in their triangle is a common neighbor of $u$ and $v$, thus a twin of either $u$ or $v$. In both cases, the same single hyperedge connects the two vertices of the triangle in one part to the third vertex in the other part, a contradiction. Thus each vertex of the $t$-fan is a twin of $u$, hence there are at most $k\le 2t$ vertices in the $t$-fan, a contradiction.

\smallskip

Let us turn to the proofs of the upper bounds. Assume that we are given a hypergraph $\cH$. We say that an edge $uv$ of the shadow graph is \textit{$p$-heavy} if there are at least $p$ hyperedges of $\cH$ containing both $u$ and $v$, and otherwise $uv$ is \textit{$p$-light}. If $p=2$, we omit $p$ and say that $uv$ is heavy/light. 
The main reason to distinguish heavy edges is that whenever we have a subgraph of the shadow graph of $\cH$, and we want to show that it is the core of a Berge copy, we can greedily pick for its heavy edges new hyperedges containing a them, even if we have already picked some hyperedges to represent other edges. We will use a form of this observation from \cite{gerb}.

\begin{proposition}[\cite{gerb}]\label{trivi}
Assume that $p\ge |E(F)|$ 
and we find a Berge copy of a subgraph $F'$ of $F$ in $\cH$, such that its core is extended to a copy of $F$ with $p$-heavy edges in the shadow graph $G$. Then this copy is the core of a Berge-$F$ in $\cH$.
\end{proposition}

 Let $\cH_1$ denote the subhypergraph consisting of the hyperedges which contain a triangle $T$ with two or three heavy edges, such that one of the edges is $p$-light. We let $G_1$ denote the graph formed by the edges of the triangles $T$ for each hyperedge in $\cH_1$. 
 Let $\cH_2$ denote the subhypergraph consisting of the hyperedges such that each of their subedges is $p$-heavy, and $G_2$ denote the shadow graph of $\cH_2$. Let $\cH_3$ denote the subhypergraph consisting of the other hyperedges. 
 
 Let us observe that the heavy edges of a hyperedge $h$ of $\cH_3$ form cliques. Indeed, if a heavy connected component is not a clique, then there are adjacent edges $uv$ and $vw$ of that component such that $uw$ is not in that component, i.e. $uw$ is not heavy. But then $h$ is in $\cH_1$, a contradiction. Each of these cliques either has order 2, or has only $p$-heavy edges. Indeed, otherwise we can pick a subtriangle containing a $p$-light edge, thus $h$ in in $\cH_1$, a contradiction. 
 Observe that these heavy cliques each have size at most $r-1$, as otherwise $h$ would be in $\cH_2$. Let $G_3$ be the graph formed by the light edges in the hyperedges of $\cH_3$.

\begin{proposition}\label{fre}
 If $\cH$ is a Berge-$F$-free $r$-uniform hypergraph and $p=|E(F)|$, then we have the following statements.
 
 \textbf{(i)} $G_2$ is $F$-free.
 
 \textbf{(ii)} There are at most $(|E(F)|-1)\cN(K_3,G_1)+\cN(K_r,G_2)+|E(G_3)|/k(r-k)$ hyperedges in $\cH$, where $k$ is $\min\{r-1,|V(F)|-1\}$.
\end{proposition}

\begin{proof}
Observe that \textbf{(i)} follows from Proposition \ref{trivi}.

For each hyperedge of $\cH_1$, we picked a triangle in $G_1$, and such a triangle contains an $|E(F)|$-light edge, thus was counted at most $|E(F)|-1$ times. This shows that $\cH_1$ has at most $(|E(F)|-1)\cN(K_3,G_1)$ hyperedges. 
For each hyperedge of $\cH_2$, we picked a $K_r$ in $G_2$, and such a clique was counted only once. This shows that $\cH_2$ has at most $\cN(K_r,G_2)$ hyperedges. 

Finally, consider a hyperedge $h$ of $\cH_3$. Recall that the heavy edges of $h$ form vertex-disjoint cliques, each of order at most $r-1$. By Proposition \ref{trivi}, those cliques must have order less than $|V(F)|$, thus order at most $k$. 
We obtained that in $h$, we have two heavy cliques of size at least 1 and at most $k$, and the edges between the cliques are light. In particular, there are at least $k(r-k)$ light edges. 
Such edges are counted only once, thus $\cH_3$ has at most $|E(G_3)|/k(r-k)$ hyperedges, completing the proof of \textbf{(ii)}.
\end{proof}

For $B_t$, we can say more about these auxiliary graphs.
 
 \begin{proposition}\label{runipr}
 If $\cH$ is Berge-$B_t$-free and $p=|E(B_t)|=2t$, then we have the following statements.
 
 \textbf{(i)} $G_1$ is $B_{3rt}$-free.

\textbf{(ii)} $G$ is $K_{1,r-1,5r^2t}$-free.

 \end{proposition}
 
 \begin{proof}
Every triangle in $G_1$ is the core of a Berge triangle in $\cH$ by Proposition \ref{trivi} and the definition of $G_1$, thus Proposition \ref{largebook} implies \textbf{(i)}.

Assume indirectly that $G$ contains a $K_{1,r-1,5r^2t}$ with parts $u$ and $v_1,\dots,v_r$ and $w_1,\dots,w_{5r^2t}$. For each $w_j$, pick a hyperedge $h_j$ containing $u$ and $w_j$. If that hyperedge does not contain $v_i$, we place $w_j$ into a set $U_i$. Then every $w_j$ is in at least one set $U_i$, hence there is an $i$ with $|U_i|\ge 5r^2t/(r+1)\ge 3rt+r$. We pick a hyperedge $h$ containing $u$ and $v_i$, this contains at most $r-2$ vertices from $U_i$. Consider a $B_{3rt}$ with rootlet vertices $u,v_i$ and page vertices from $U_i$ not in $h$. For each such page vertex $w_j$, there is a hyperedge $h_j$ containing $u$ and $w_j$ and not containing $v_i$. There exists a hyperedge $h'_j$ containing $v_i$ and $w_j$, different from $h$ since $h$ does not contain $w_j$ and different from $h_j$ since $h_j$ does not contain $v_i$. Therefore, every triangle of this $B_{3rt}$ is the core of a Berge triangle in $\cH$, thus Proposition \ref{largebook} gives a contradiction, completing the proof.
\end{proof}

Now we are ready to prove Theorem \ref{runi}. Recall that it states that $\ex_r(n,\textup{Berge-}B_t)=n^2/4k(r-k)+o(n^2)$.

\begin{proof}[Proof of Theorem \ref{runi}] The lower bound is given by Construction 4. For the upper bound, let $\cH$ be a Berge-$B_t$-free $n$-vertex $r$-uniform hypergraph and apply Proposition \ref{runipr}. There are at most $(2t-1)\cN(K_3,G_1)+\cN(K_r,G_2)+|E(G_3)|/k(r-k)$ hyperedges in $\cH$, and $G_1,G_2$ both avoid $B_{tr}$. Then by Proposition \ref{alsh} and Corollary \ref{alshcor}, they contain $o(n^2)$ triangles and $o(n^2)$ copies of $K_r$. As $G_3$ is $K_{1,r-1,(r+1)(t+(4t-3)(r-2))}$-free, the Erd\H os-Stone-Simonovits theorem implies $|E(G_3)|\le n^2/4+o(n^2)$, completing the proof.
\end{proof}

We remark that we could have obtained the same result by using Theorem \ref{luwa} to bound $|E(G_3)|$ instead of \textbf{(ii)} of Proposition \ref{runipr} and the Erd\H os-Stone-Simonovits theorem. 

\smallskip

Let us turn to $t$-fans. The situation is somewhat more complicated here, as the analogue of \textbf{(i)} of Proposition \ref{runipr} holds only in the case $r=3$. Indeed, if a triangle $uvw$ of $G_1$ contains two edges from the same hyperedge $h$, then $h$ contains $\{u,v,w\}$ and at least two edges of this triangle, say $uv$ and $vw$ are heavy. Then we can pick $h$ for $uw$, another hyperedge $h'$ for $uv$ and $h''$ for $vw$ containing the corresponding edges. If $r=3$, then $h'$ and $h''$ must be distinct, showing that $uvw$ is the core of a Berge triangle, thus Proposition \ref{largebook} shows that $G_1$ is $F_{3rt}$-free. However, if $r>3$, then it is possible that 2 hyperedges contain $\{u,v,w\}$, creating a triangle in $G_1$ that is not the core of a Berge triangle.

We remark that the 3-uniform case of Theorem \ref{kfan} easily follows. 
Indeed, $G_1$ has $O(n)$ triangles by Proposition \ref{alsh2} and the argument above. $G_2$ has $O(n)$ triangles by Proposition \ref{alsh2} and \textbf{(i)} of Proposition \ref{fre}. $G_3$ has $n^2/4+o(n^2)$ edges by Theorem \ref{luwa} and the fact that $F_t$ is 3-chromatic. Therefore, \textbf{(ii)} of Proposition \ref{fre} gives the desired bound.

\smallskip

Let us consider now uniformity $r$. Gerbner \cite{gerb} showed that if a graph $F$ is a subgraph of a blow-up of another graph $H$ and $r\ge |V(F)|$, then for any Berge-$F$-free $r$-uniform $n$-vertex hypergraph $\cH$ there is a set $S$ of $o(n^2)$ edges in the shadow graph of $\cH$ such that every copy of $H$ in the shadow graph contains an edge from $S$. Here we need to extend this to smaller values of $r$. Fortunately, the assumption $r\ge |V(F)|$ was used in \cite{gerb} only to obtain that $\cH$ has $O(n^2)$ hyperedges (due to a result from \cite{gp1}). This conclusion holds when $F=F_t$ in every uniformity by Proposition \ref{genturfan}. Thus most of the proof of the following statement can be found in \cite{gerb}. We provide a proof for sake of completeness.

\begin{proposition}\label{remo}
If $\cH$ is Berge-$F_t$-free, then there is a set $S$ of edges in the shadow graph of $\cH$ such that any triangle in the shadow graph contains an edge from $S$ and $|S|=o(n^2)$.
\end{proposition}

\begin{proof} Observe that $\cH$ has $O(n^2)$ hyperedges by Lemma \ref{celeb} and Proposition \ref{genturfan}.
We distinguish two types of triangles in the shadow graph of $\cH$. We have triangles that are inside some hyperedge of $\cH$. They can be counted by picking a hyperedge $O(n^2)$ ways, and then picking the vertices $\binom{r}{3}$ many ways. Therefore, there are $O(n^2)$ such hyperedges.

For each hyperedge of $\cH$, we pick a sub-edge randomly with uniform distribution, independently from the other hyperedges. Let $G'$ be the graph having those edges. Then $G'$ is clearly $F_t$-free. Consider a triangle that shares at most two vertices with any hyperedge of $\cH$. Every edge of this triangle is in at least one hyperedge of $\cH$, thus it is in $G'$ with probability at least $1/\binom{r}{2}$. Distinct edges of the triangle are in $G'$ independently of each other, as they may be included in $G'$ only via distinct hyperedges.  Therefore, a triangle 
is in $G'$ with probability at least $1/\binom{r}{2}^3$. 
This implies that the number of the copies of the second type of $H$ is at most $\binom{r}{2}^3\cN(K_3,G')\le \binom{r}{2}^3\ex(n,K_3,F_t)=O(n)$.

We obtained that the number of copies of $H$ is $o(n^3)$, thus we can apply the removal lemma to obtain the desired edge set $S$.
\end{proof}

Now we are ready to prove Theorem \ref{kfan}. Recall that it states that $\ex_r(n,\textup{Berge-}F_t)=n^2/4k(r-k)+o(n^2)$.

\begin{proof}[Proof of Theorem \ref{kfan}] 
Let $\cH$ be a Berge-$F_t$-free $r$-uniform $n$-vertex hypergraph. We partition $\cH$ to $\cH_1$, $\cH_2$ and $\cH_3$ as described earlier.
We partition $\cH_1$ to two parts. We let $\cH_1'$ denote the hyperedges that have a triangle consisting of three edges that are each contained in exactly two hyperedges (i.e. 2-heavy, 3-light edges). Let $\cH_1''$ consist of the other hyperedges in $\cH_1$.

To bound the number of hyperedges in $\cH_1'$, we use Proposition \ref{remo}.
Each hyperedge of $\cH_1'$ contains a 3-light edge from $S$, and such an edge is counted at most twice. Thus there are at most $2|S|=o(n^2)$ hyperedges in $\cH_1'$.

Recall that the hyperedges of $\cH_1''$ each contain a triangle with 2 or 3 heavy edges and at least one $p$-light edge. Let us pick an arbitrary such triangle for each hyperedge of $\cH_1''$ and let $G_1''$ denote the graph consisting of the edges of these triangles. We will show that $G_1''$ is $F_{3rt}$-free. Consider a triangle $uvw$ in $G_1''$, we will show that it is the core of a Berge triangle in $\cH$. Assume first that $uv$ is light, it is contained in the hyperedge $h$. Then $uw$ and $vw$ are not light, thus they are contained in hyperedges $h'$ and $h''$ respectively with $h'\neq h$ and $h''\neq h$. If $h'= h''$, then $h'$ contains $uv$, thus $uv$ is not light, a contradiction. Thus $h'\neq h''$ and $uvw$ is the core of a Berge triangle. Assume now that each edge of $uvw$ is heavy, then one of them, say $uv$ is 3-heavy. Then we can pick the hyperedges for the edges greedily. We pick an arbitrary hyperedge $h$ containing $uw$, then an arbitrary hyperedge $h'\neq h$ containing $vw$, and an arbitrary hyperedge distinct from $h$ and $h'$ containing $uv$. Now we can apply Proposition \ref{largebook} to show that $G_1''$ is $F_{3rt}$-free. Together with Proposition \ref{alsh2}, this shows that $G_1''$ has $O(n)$ triangles. Each hyperedge of $\cH_1''$ contains such a triangle, and that triangle is contained in at most $m-1$ hyperedges, thus $\cH_1''$ has $O(n)$ hyperedges.

We obtained that $\cH_1$ has $o(n^2)$ hyperedges. Clearly $\cH_2$ has at most $\cN(K_r,G_2)$ hyperedges. As $G_2$ is $F_t$-free by \textbf{(i)} of Proposition \ref{fre}, it has $O(n)$ copies of $K_r$ by Proposition \ref{genturfan}. $G_3$ has $n^2/4+o(n^2)$ hyperedges by Theorem \ref{luwa}. Therefore, \textbf{(ii)} of Proposition \ref{fre} completes the proof.
\end{proof}

Our methods also let us to prove the following result.

\begin{thm}\label{bena} For any $F$ and $r$ we have
\[\ex_r(n,\textup{Berge-}F)\le (|E(F)|-1)\ex(n,K_3,F)+\ex(n,K_r,F)+(|E(F)|-1)\ex(n,F)/k(r-k)+o(n^2),\] 

where $k$ is $\min\{r-1,|V(F)|-1\}$.
\end{thm}

Note that this theorem gives an improvement on Lemma \ref{celeb} if $(|E(F)|-1)<k(r-k)$ and the terms $(|E(F)|-1)\ex(n,K_3,F)$ and $o(n^2)$ are both negligible, i.e. $\ex(n,K_3,F)=o(n^2)$ and $F$ has chromatic number at least 3.

\begin{proof}
We slightly modify the definitions of the subhypergraphs. We let $\cH_1'$ be the subhypergraph consisting of the hyperedges which contain a triangle with two $p$-heavy and one $p$-light edges, where $p=|E(F)|$. We pick such a triangle for each hyperedge of $\cH_1$, and let $G_1'$ be the graph of the edges in these triangles. We keep the definition of $\cH_2$ and we let $\cH_3'$ consist of the remaining hyperedges. Let $G_3'$ be formed by the $p$-light edges in hyperedges of $\cH_3'$.

We claim that $G_1'$ is $F$-free. For each $p$-light edge $uv$ we mark one of the hyperedges where we picked a triangle containing $uv$. Observe that we marked each hyperedge at most once, as we picked only one $p$-light edge for each hyperedge. Assume now that there is a copy of $F$ in $G_1$. Then we first pick for each $p$-light edge the marked hyperedge. This can be extended to a Berge-$F$ by Proposition \ref{trivi}, a contradiction.

Now we proceed as in \textbf{(ii)} of Proposition \ref{fre}. Clearly $\cH_1'$ has at most $(|E(F)|-1)\cN(K_3,G)$ hyperedges and $\cH_2$ still has at most $\cN(K_r,G_2)$ hyperedges. 
Let $h$ by a hyperedge of $\cH_3'$. Then the $p$-heavy edges of $\cH$ form vertex-disjoint cliques of order at most $k$, thus there are at least $k(r-k)$ $p$-light edges in $h$. Such edges are counted at most $p-1$ times in $G_3'$, thus there are at most $(p-1)|E(G_3)|/k(r-k)$ hyperedges in $\cH_3'$. Theorem \ref{luwa} implies that $|E(G_3)|\le \ex(n,F)+o(n^2)$, completing the proof.
\end{proof}
 
In the 3-uniform case, we introduce a coloring on the edges of the shadow graph. This is similar to Lemma \ref{celeb2}, but our goal is to obtain an upper bound where the number of blue triangles may matter more (as there are only $o(n^2)$ of them), but the number of red edges should be divided by 2. In order to achieve that, we want to place two red edges into most of the hyperedges. 

We pick two light edges from each hyperedge of $\cH_3$ and color them red. We take all three edges from each hyperedge of $\cH_1$ and $\cH_2$ and color them blue. Note that red edges are light, but blue and edges may be heavy or light. 
 
Let $G$ denote the shadow graph of $\cH$, $G_r$ denote the graph of the red edges and $G_b$ denote the graph of the blue edges.

\begin{proposition}\label{regi}
There is no $B_{12t+1}$ in $G$ such that on each page, at least one of the edges is blue.
\end{proposition}

\begin{proof}
Consider a triangle $uvw$ in $G$. We claim that $uvw$ is the core of a Berge triangle unless $\{u,v,w\}$ is a hyperedge with at least 2 light edges. Indeed, consider the auxiliary bipartite graph with the edges $uv$, $vw$, $uw$ as vertices in part $A$, and the hyperedges of $\cH$ as vertices in part $B$, with an edge connected to a hyperedge if and only if the hyperedge contains the edge. A matching covering $A$ is equivalent to a Berge triangle. By Hall's matching theorem, if there is no such matching, then we have either tree vertices of $A$ connected to exactly two vertices of $B$, or tree vertices of $A$ connected to exactly one vertex of $B$, or two vertices of $A$ connected to exactly one vertex of $B$. The first possibility means two hyperedges containing $u,v,w$, impossible because of uniformity 3. The second possibility means that the hyperedge $\{u,v,w\}$ contains 3 light edges, and the third possibility means that the hyperedge $\{u,v,w\}$ contains 2 light edges.

Hyperedges with 3 light edges do not contain a blue edge. Hyperedges with two light edges may contain a single blue edge, but if we have a $B_{12t+1}$ with a blue edge on each page that contains such a triangle, then the rootlet edge of the book is a light edge. This means that only one of the triangles in the $B_{12t+1}$ can have the property that it has a blue edge on each page and is not the core of a Berge triangle. Then we have a $B_{12t}$ without such triangles and Proposition \ref{largebook} gives a contradiction completing the proof.
\end{proof}
 
 Now we are ready to prove Theorem \ref{3uni}.
 Recall that it states $\ex_3(n,\textup{Berge-}B_t)=\lfloor n^2/8\rfloor+(t-1)^2$ for $t>2$ and $\ex_3(n,\textup{Berge-}B_2)=\lfloor n^2/8\rfloor$.
 
 \begin{proof}[Proof of  Theorem \ref{3uni}] The lower bound is by Constructions 1 and 2 for $t=2$ and by Construction 3 for $t\ge 3$. Let us turn to the proof of the upper bound.
 Let us choose $\varepsilon>0$ sufficiently small and assume that $n$ is large enough. By the Erdős-Simonovits stability theorem, there is a $\delta$ such that if $G$ is an $n$-vertex $B_t$-free graph with at least $n^2/4-\delta n^2$ edges, then $G$ can be turned into a complete bipartite graph by adding and removing at most $\varepsilon n^2$ edges. Here we use that $B_t$ has chromatic number 3. 
We can also assume that $\delta$ is small enough (with respect to $t$, $r$ and $\varepsilon$). 
If $n$ is large enough, then $G_b$ has at most $\delta n^2$ triangles by Proposition \ref{alsh}. Observe that this completes the proof unless $G$ has at least $n^2/4-\delta n^2/2$ edges.

First we remove the vertices with degree at most $\varepsilon^{1/3} n/10$ one by one, to obtain the subgraph $G'$. If we remove $k$ vertices and $k\ge 10\delta n/\varepsilon$, then $G'$ has at most $(n-k)^2/4$ edges by a theorem of Simonovits mentioned in Section 2, thus $G$ has at most $(n-k)^2/4+k\varepsilon^{1/3} n/10<n^2/4-\delta n^2$ edges (using that $\delta$ is small enough), a contradiction.
Clearly $G'$ has at least $|V(G')|^2/4-\delta n^2$ edges, thus we can apply the Erd\H os-Simonovits stability theorem to obtain a bipartition $V(G')=A\cup B$. We pick the bipartition with the least number of edges inside parts, which implies that for each vertex $v\in A$, at least half of its neighbors are in $B$. In particular, each vertex of $A$ has at least $\varepsilon^{1/3} n/20$ neighbors in $B$. If $|B|\le 2n/5$, then $G'$ has at most  $6n^2/25+\delta n^2/\varepsilon^{2/3}<(n-k)^2/4-\delta n^2$ edges, a contradiction.

Let $A_1$ denote the vertices of $A$ with at least $|B|-\varepsilon^{1/3} n/40+5r^2t$ neighbors in $B$, and $A_2$ denote the set of other vertices in $A$. Recall that $G$ is $K_{1,2,5r^2t}$-free by Proposition \ref{runipr}. 
This, together with Proposition \ref{regi} and the fact that each vertex of $A$ has at least $\varepsilon^{1/3} n/20$ neighbors in $B$, imply that vertices of $A_1$ do not have a common neighbor in $A$. 
In particular, $G[A_1]$ is a matching in $G[A]$, and vertices of $A_2$ each have at most one neighbor in $A_1$, thus they have degree at most $|A_2|$ in $G[A]$. Observe that for a red edge in $G[A_1]$, at most $t-1$ vertices of $B$ are connected to both its endpoints by red edges.

Let $m_1$ denote the number of vertices not in $G'$, $m_2$ denote the number of edges in $G[A]$ incident to a vertex of $A_2$ and $m_3$ denote the number of red edges in $G[A_1]$. Let $E_1$ denote the set of blue edges in $G[A_1]$ such that at most $24t$ vertices of $B$ are connected to both endpoints,  and $m_4$ denote their number. Let $E_2$ denote the set of the other blue edges in $G[A_1]$ and $m_5$ denote their number. Let $E_1',E_2'$ denote the analogous edge sets and $m_2',m_3',m_4',m_5'$ denote the analogous quantities obtained by considering $G[B]$ instead of $G[A]$.

We will examine how these quantities add more blue triangles by decreasing the number of red edges compared to the mono-red 
$K_{\lfloor n/2\rfloor,\lceil n/2\rceil}$.
The $m_1$ low degree vertices each represent losing at least $n/3$ red edges. On the other hand, a vertex $v$ of degree $d$ is in at most $6dt$ blue triangles. Indeed, in the blue neighborhood of $v$ there is no vertex of degree $12t+1$ by Proposition \ref{regi}. 

If $u\in A$ and $v\in B$ and $uv\not\in E(G)$, then we say that $uv$ is a \textit{missing edge}. If $uv$ is not a red edge, then we say that $uv$ is a \textit{missing red edge}.
As the vertices of $A_2$ have degree at most $|A_2|$ in $G[A]$, we have $m_2\le |A_2|^2$. By the definition of $|A_2|$, at least $|A_2|\varepsilon^{1/3} n/41\ge \sqrt{m_2}\varepsilon^{1/3} n/41$ edges are missing between $A_2$ and $B$. As $m_2\le \varepsilon n^2$, at least $m_2/41\varepsilon^{1/6}$ edges are missing between $A_2$ and $B$. Recall that for each red edge of $G[A_1]$, its endpoints have at most $t-1$ common red neighbors. Moreover, at most $12t$ vertices are connected to both endpoints such that one or both the connecting edges are blue. Thus at least $m_3|B|/2-(13t-1)\ge m_3n/6$ edges are  missing between $B$ and the endpoints of red edges of $G[A_1]$. Similarly, by the definition of $m_4$, at least $m_4|B|/2-24t\ge m_3n/6$ edges are missing between $B$ and the endpoints of edges of $E_1$. The analogous statements hold for $m_2',m_3',m_4'$. 

Observe that the missing edges between $A$ and $B$ may be counted twice. Still, the number of missing  edges is more than twice the number of blue triangles found plus the number $m_3$ of additional red edges. More precisely, if $m_5=m_5'=0$, then the number of red edges is at most 
\begin{equation}\label{eq1}
n^2/4-m_1 n/3-m'_2/82\varepsilon^{1/6}-m_2/82\varepsilon^{1/6}+m_3+m'_3-m_3n/12-m'_3n/12-m_4n/12-m_4'n/12,\end{equation}
while the number of blue triangles is at most 
\begin{equation}\label{eq2}
   6m_1t\varepsilon^{1/3} n/20+m_212t+m'_212t+m_412t+m'_412t. 
\end{equation} Indeed, every blue edge is in at most $12t$ blue triangles, and every blue triangle contains at least one edge from $G[A]$ or $G[B]$. Clearly, $|E(G_r)|/2+\cN(K_3,G_b)\le n^2/8$, completing the proof.

If $m_5$ or $m_5'$ is greater than zero, than the above argument still holds, but we need to add the additional blue triangles containing an edge of $E_2$ or $E_2'$. Consider first those triangles where the third vertex is not incident to $E_2\cup E_2'$. Then we  have two blue edges in the triangle. Observe that we have only counted edges missing between $A$ and $B$ so far, thus we did not subtract those. This means that for each blue triangle, we lose two red edges, the sum $\cN(K_3,G_b)+|E(G_r)|/2$ does not change. However, in the case the third vertex of the triangle is incident to an edge in $E_2\cup E_2'$, we may count these edges twice, thus we need to deal with those triangles in a different way.


Let $uv\in E_2$, $u'v'\in E_2'$. As $u$ is connected only to $v$ in $G[A]$ and $u'$ is connected only to $v'$ in $G[B]$, every triangle containing $uu'$ also contains either $v$ or $v'$. This means that every hyperedge containing $uu'$ is inside $P:=\{u,u',v,v'\}$.

Let us consider what sets of hyperedges can be inside $P$. If there are 4 hyperedges in $P$, then there are 4 blue triangles and four blue edges between $A$ and $B$ inside $P$. If there are 3 hyperedges in $P$, then there are 4 blue triangles and 4 blue edges. Two hyperedges in $P$ can appear in two different ways, either they share the edge in $E_2\cup E_2'$, or not. In the first case, there are no blue triangles. In the second case, we have 2 blue triangles and 3 blue edges and a missing edge.
If there is 0 or 1 hyperedge on those 4 vertices, then we do not have blue triangles.

Let us consider the case $t=2$. Let us examine first the sets $P:=\{u,u',v,v'\}$ with 4 hyperedges in $P$. If a vertex $x$ not in $P$ is connected in $G$ to two vertices of $P$, say $u$ and $v$, then we can find a $B_2$. Indeed, if $ux$ and $vx$ are in the same hyperedge $h$, then $h$ corresponds to $uv$, $\{u,v,u'\}$ corresponds to $uu'$, $\{u,v,v'\}$ corresponds to $vv'$, $\{u,u',v'\}$ corresponds to $uv'$ and $\{v,u',v'\}$ corresponds to $vu'$. If $ux$ is in a hyperedge $h_1$ and $vx$ is in a different hyperedge $h_2$, then we extend this correspondence with $\{u,v,u'\}$ corresponding to $uv$, $\{u,u',v'\}$ corresponding to $uv'$ and $\{v,u',v'\}$ corresponding to $vv'$.

This means that for each set of four blue triangles of this type (let $m_6$ be their number), there are at least $n-4$ missing edges incident to vertices in $P$. Observe that any edge $uv\in E_2$ is in at most $6t$ 4-sets of this type. Indeed, there is a Berge triangle with core $uvu'$ and another with core $uvv'$, thus we can apply Proposition \ref{largebook}. This means that there are at least $m_6(n-4)/6t$ missing edges incident to vertices contained in 4-sets of this type.
Note that we may have counted these missing edges once already, because of their other endpoints.

If there are 3 hyperedges on $P$, without loss of generality they are $\{u,v,u'\}$ containing $uu'$, $\{u,v,v'\}$ containing $uv'$ and $\{u,u',v'\}$ containing $u'v'$. If there are additional hyperedges containing $uv$ and $u'v'$, then we found a $B_2$, a contradiction. If there is a common neighbor $x\not\in P$ of $u$ and $v$ in $G$ such that $uv$ and $vx$ are from different hyperedges $h_1$ and $h_2$, then we can easily find a $B_2$. If $u'$ and $v'$ have 4 different common neighbors $x_1,x_2,x_3,x_4$ such that $u'x_i$ and $v'x_i$ are from different hyperedges, then consider the hyperedges containing $u'x_1$ and $v'x_1$. They avoid at least one of the other neighbor $x_2$, thus two other hyperedges contain $u'x_2$ and $v'x_2$. These four hyperedges with $\{u,u',v'\}$ form a $B_2$, a contradiction. We obtained first that at least one of $uv$ and $u'v'$ is not contained in any other hyperedge, thus has the property that the common neighbors of their endpoints are connected to them by different hyperedges. Therefore, the endpoints of that edge have at most 3 common neighbors. This means that at least $n/3$ red edges are missing between $A$ and $B$ that are incident to $P$. Let $m_7$ be the number of sets $P$ with three triangles. Observe that there is a Berge triangle with core $uvu'$ or $uvv'$, and a Berge triangle with core $u'v'v$ or $u'v'u$. This implies that any edge $uv\in E_2$ or $u'v'\in E_2'$ is in at most $12t$ 4-sets of this type. Indeed, otherwise we can find $12t$ triangles sharing the edge $uv$ or $u'v'$, thus we can apply Proposition \ref{largebook}. This means that there are at least $m_7n/36t$ missing edges incident to vertices contained in 4-sets of this type.

Let $m_8$ be the number of sets $P$ with two blue triangles. Recall that there are 3 blue edges inside $P$. An edge $uv\in E_2\cup E_2'$ can be in at most two such sets. Indeed, without loss of generality $\{u,v,u'\}$ and $\{u,u',v'\}$ are hyperedges of $\cH$, these hyperedges can be used for the edges $vu'$ and $uu'$. Similarly from the set $\{u,v,u'',v''\}$ we have distinct hyperedges containing, say, $uu''$ and $vv''$, and from the set $\{u,v,u''',v'''\}$ we have another hyperedge containing $uv$. We found a Berge-$B_2$, a contradiction.

Observe that an edge $uv\in E_2\cup E_2'$ can be in sets $P$ with 2,3 or 4 blue triangles at the same time. We count the blue edges as missing red edges inside sets with two blue triangles only when we consider $m_8$. This means that there are at least $3m_8$ missing red edges incident to vertices in such sets, and these missing red edges have not been counted elsewhere. The missing edges incident to $uv$ have been counted at most 4 times: twice because $uv$ can be in sets with 3 or 4 blue triangles, and twice because we may have counted the missing edge at its other endpoint already.

Now we can extend our equations (\ref{eq1}) and (\ref{eq2}): the number of red edges is at most $n^2/4-m_1 n/3-m'_2/82\varepsilon^{1/6}-m_2/82\varepsilon^{1/6}+m_3+m'_3-m_3n/12-m'_3n/12-m_4n/12-m_4'n/12-m_6(n-4)/24t-m_7n/144t-3m_8$, while the number of blue triangles is at most $6m_1t\varepsilon^{1/3} n/20+m_212t+m'_212t+m_412t+m'_412+4m_6+3m_7+2m_8$. Clearly, $|E(G_r)|/2+\cN(K_3,G_b)\le n^2/8$, completing the proof in the case $t=2$.

Let us assume now that $t\ge 3$.
Let us consider an auxiliary bipartite graph $H$. The two parts of $H$ are $E_2$ and $E_2'$. Two vertices $uv\in E_2$ and $u'v'\in E_2'$ are connected by two blue edges if there are 4 hyperedges on $P:=\{u,u',v,v'\}$ in $\cH$,  let $a$ be the number of such sets. Let $b$ be the number of sets $P$ with 3 hyperedges, and connect the pair with one blue edge in $H$. Let $c$ be the number of sets $P$ with 2 hyperedges, and connect the pair again with one blue edge in $H$. Let $d$ be the number of sets $P$ with 4 red edges in $G$, then we add a directed red edge from $uv$ to $u'v'$ if $u'v'$ is contained in the corresponding hyperedges, and from $u'v'$ to $uv$ if $uv$ is contained in the corresponding hyperedges. 
For a vertex $uv$ of $H$, we denote by $d(uv)$ the number of blue edges incident to it plus the number of outgoing red edges.

\begin{clm}\label{claim2}
For every vertex $uv$ of $H$, we either have $d(uv)<t$, or $uv$ is not incident to any blue edge in $H$. 
\end{clm}

\begin{proof}[Proof of Claim] Let us assume that $uv$ is connected to $u_1v_1, \dots,u_kv_k$ by two blue edges, to $w_1x_1,\dots, w_\ell x_\ell$ by one blue edge, and a red edge goes from $uv$ to each of $y_1z_1,\dots, y_mz_m$. First we will try to find a $B_t$ in $H$, with rootlet edge $uv$ and page vertices $u_1,v_1,\dots,u_k,v_k$, $w_i$ or $x_i$ for every $i\le \ell$ and $y_1,\dots,y_m$. 

The hyperedges $\{u,v,u_i\}$ correspond to the edges $uu_i$, the hyperedges $\{u,v,v_i\}$ correspond to the edges $vv_i$, the hyperedges $\{u,u_i,v_i\}$ correspond to the edges $uv_i$, and the hyperedges $\{v,u_i,v_i\}$ correspond to the edges $vu_i$. Observe that a blue edge in $H$ connecting $uv$ to $w_ix_i$ means that in $G$ we have two blue triangles on these four vertices and in $\cH$ we have a hyperedge containing $uv$ and, say, $w_i$, and a hyperedge containing $w_ix_i$ and, say, $u$. Then $\{u,v,w_i\}$ corresponds to $vw_i$ and $\{u,w_i,x_i\}$ corresponds to $uw_i$. The red edge from $uv$ to $y_iz_i$ means that $\{y_i,z_i,u\}$ and $\{y_i,z_i,v\}$ are in $\cH$; the first one corresponds to $uy_i$ and the second one corresponds to $vy_i$. 

It is left to find a hyperedge corresponding to the edge $uv$.
By the definition of $E_2$ and $E_2'$, $u$ and $v$ have at least $24t$ common neighbors. They have at most $12t$ common blue neighbors by Proposition \ref{regi}, thus they have at least $12t$
common red neighbors $q_1,\dots,q_{12t}$.

If there is a hyperedge $\{u,v,q_i\}$ with $1\le i\le 12t$, then that corresponds to the edge $uv$, completing the proof. If there is no such hyperedge, then we find another $B_t$. For each $i$, a hyperedge $h$ contains $uq_i$, a hyperedge $h'$ contains $vq_i$ and a hyperedge $h''$ contains $uv$. These are three different hyperedges, thus $uvq_i$ is the core of a Berge triangle, hence Proposition \ref{largebook} completes the proof.
\end{proof}

Let $H'$ denote the subgraph of $H$ we obtain by deleting the vertices from $H$ that are not incident to any blue edges in $H$. Let $A'$ denote the part of $H'$ corresponding to edges inside $A$, $B'$ denote the other part, and let $m_9$ and $m_9'$ denote the number of vertices in $A'$ resp. $B'$. Assume without loss of generality that $m_9\le m_9'$. 

The number of edges in $H'$ is $4a+2b+2c+d$, thus
we have $4a+2b+2c+d\le (m_9+m_9')(t-1)$ by Claim \ref{claim2}. The number of edges incident to part $A'$ in $H'$ is at least $2a+2b+c$, thus we also have $2a+2b+c\le m_9(t-1)$ by Claim \ref{claim2}.

Let $E_3$ denote the set of edges with both endvertices incident to $E_2\cup E_2'$. The only blue triangles left to count are those with all three edges in $E_3$. Recall that we obtained that the number of blue triangles of other types was at most twice the number of missing red edges counted earlier. 
We will consider the missing red edges in $E_3$, those we did not count earlier. Let $m_{10}$ be their number. This means that it is enough to prove that the number of blue triangles in $E_3$ is at most $m_{10}/2+(t-1)^2$.

We have $4a+4b+2c$ blue triangles in $E_3$, but at least $b$ of them does not correspond to a hyperedge. This means that  it is enough to prove that $4a+3b+2c- m_{10}/2\le (t-1)^2$.

Assume first that $m_9\ge t-1$. We claim that 
$(m_9+m_9')(t-1)-m_9m_9'\le (t-1)^2$. Indeed, the inequality trivially holds with equality if 
$t-1=m_9$. If we decrease $t$ by one, then the left hand side decreases by $m_9+m_9'$ and the right hand side decreases by $2t-3\le m_9+m_9'$, proving the claimed inequality.
We have that $m_{10}\ge 4a+4b+4c+2(m_9m_9'-a-b-c-d)=2m_9m_9'+2a+2b+2c-2d$. Indeed, for $a+b+c$ pairs of edges $uv\in A'$, $u'v'\in B'$, there are four missing red edges inside $\{u,v,u',v'\}$, for $d$ pairs of such edges there are no missing red edges inside $\{u,v,u',v'\}$, and for all the other such pairs there are at least two missing red edges inside $\{u,v,u',v'\}$. Thus we have $4a+3b+2c-m_{10}\le 4a+3b+2c-(2m_9m_9'+2a+2b+2c-2d)/2=3a+2b+c+d-m_9m_9'\le (m_9+m_9')(t-1)-m_9m_9'\le (t-1)^2$.

If $m_9<t-1$, then we use the simple bound $m_{10}\ge 4a+4b+4c$. Then $4a+3b+2c-m_{10}/2\le 2a+b\le m_9(t-1)<(t-1)^2$.
 \end{proof}

\section{Concluding remarks}

With a more careful application of our methods, one can
improve the upper bound of Theorem \ref{runi} to $\ex_r(n,\textup{Berge-}B_t)\le n^2/4k(r-k)+O(n)$. Inside each hyperedge we either color at least $k(r-k)$ light edges to red, or a triangle with at least 2 heavy edges to blue. Observe that every bue triangle is the core of a Berge triangle. Following the proof of Theorem \ref{3uni}, we obtain that every degree in $G[A_1]$ is less than $r-1$, thus there are $O(n)$ edges inside $A_1$. This, together with Proposition \ref{largebook} implies that there are $O(n)$ blue triangles with an edge inside $A_1$. The rest of the blue triangles can be handled similarly as in the proof of Theorem \ref{3uni}.

To obtain an exact result, we need to find the best way to improve Construction 4. It is unclear in general, but the straightforward generalizations of Construction 3 contain $B_2$, thus we can have the following conjecture.

\begin{conjecture} If $r\le 6$ and $n$ is large enough, then
$\ex_r(n,\textup{Berge-}B_2)= n^2/8(r-2)$.
\end{conjecture}

If $t>r-3$, i.e. $k=r-1$, then we can improve Construction 4 by $(t-1)^2$, the same way as Construction 3 improves Constructions 1 and 2. We partition the right side to twins the same way as the left part. We take the first $t-1$ twins from the left part and another vertex from the right part, and the first $t-1$ twins from the right part and another vertex from the the left part that is not among the first $t-1$ twins. Then we take one of the first $t-1$ twins in the left part, one of the first $t-1$ twins on the right part, and add one of the missing hyperedges.

\begin{conjecture} If $t>r-2$ and $n$ is large enough, then
$\ex_r(n,\textup{Berge-}B_2)= n^2/4(r-1)+(t-1)^2$.
\end{conjecture}

\smallskip

We remark that in Proposition \ref{fre} and Theorem \ref{bena}, we showed that heavy edges form vertex-disjoint cliques of order at most $k$, thus there are at least $k(r-k)$ light edges. In fact, this gives a better bound: there are at most $\lfloor r/k\rfloor \binom{k}{2}+\binom{r-k\lfloor r/k\rfloor}{2}$ heavy edges, thus at least $\binom{r}{2}-\lfloor r/k\rfloor \binom{k}{2}-\binom{r-k\lfloor r/k\rfloor}{2}$ light edges. However, it is an improvement only if $r>2k$. In particular, not in the range of Theorems \ref{runi} and \ref{kfan}. Moreover, it was shown in \cite{gerb} that if $\ex(n,K_3,F)=o(n^2)$ and $r> (\chi(F)-1)(|V(F)|-1)$, then $\ex_r(n,\textup{Berge-}F)=o(n^2)$. Therefore, in the cases when Theorem \ref{bena} gives a better bound than Lemma \ref{celeb} that we discussed after stating Theorem \ref{bena}, we do not get any improvement. 

\smallskip

A more careful analysis of the proof of Theorem \ref{3uni} gives a stability result. If any of $m_1,m_2,m_3,m_4$ is positive, then there are at most $n^2/8-\Omega(n)$ hyperedges. Therefore, if a $B_t$-free 3-uniform hypergraph has at least $n^2/8-o(n)$ edges, then its shadow graph is an almost balanced, almost complete bipartite graph, with independent edges added inside the parts. In particular, there is a set of independent edges such that every hyperedge contains at least one of them.

\smallskip

Gy\H ori proved a stronger result than Theorem \ref{triangle}. He showed that $\sum_{h\in\cH} |h|-2\le n^2/4$. Since that, there were several results concerning weighted size of Berge-$F$-free non-uniform hypergraphs. In particular, Gerbner and Palmer \cite{gp1} showed that if every hyperedge has order at least $|V(F)|$, then $\sum_{h\in\cH} |h|= O(n^2)$. English, Gerbner, Methuku and Palmer \cite{sgmp} showed that if every hyperedge has order at least $|V(F)|$, then $\sum_{h\in\cH} |h|^2= O(n^2)$. Moreover, there is a threshold $th(F)$ such that if every hyperedge has order at least $th(F)$, then $\sum_{h\in\cH} |h|= o(n^2)$.
The result of Gerbner and Palmer was based on the observation that for some constant $c$, we can place $c|h|$ subedges into each hyperedge $h$, such that the number of times an edge is used is bounded by a constant. The improvement by English, Gerbner, Methuku and Palmer is proved by placing $c|h|^2$ subedges into $h$. This is doable since there is no $|E(F)|$-heavy $F$ inside $h$, thus the Erd\H os-Stone-Simonovits theorem gives an upper bound on the number of heavy edges. Our arguments show that in the case $F$ is a book or a fan, there are only $c'|h|$ heavy edges inside $h$. Therefore, we can obtain sharper bounds for a specific weight.

Let $w(r)=\binom{r}{2}-\lfloor r/(|V(F)|-1)\rfloor \binom{|V(F)|-1}{2}-\binom{r-(|V(F)|-1)\lfloor r/(|V(F)|-1)\rfloor}{2}\ge \binom{r}{2}-(|V(F)|-2)r/2$ if $r>|V(F)|$ and $w(r)=r-1$ if $r\le |V(F)|$. Then our arguments can be modified to show that $\sum_{h\in\cH}w(|h|)\le n^2/4+o(n^2)$ if $\cH$ is $B_t$-free or $F_t$-free. The crucial part is Proposition \ref{largebook}, with an $r$ in the statement. It is easy to see from its proof that the same conclusion holds even if we can have smaller hyperedges. Thus we apply it with $r=th(F)-1$, then we can obtain for the hyperedges $h$ of order less than $th(F)$, that the heavy edges form cliques of order at most $|h|-1$ and at most $|V(F)|-1$ inside $h$, similarly to the uniform case (with $o(n^2)$ exceptions). Therefore, $h$ contains at least $w(|h|)$ light edges. On the other hand, the sum of the number of light edges is at most the number of edges in the shadow graph, thus Theorem \ref{luwa} implies that $\sum_{h\in\cH, \, |h|<th(F)}w(|h|)\le n^2/4+o(n^2)$. For larger hyperedges, we have $\sum_{h\in\cH, \, |h|\ge th(F)}w(|h|)< \sum_{h\in\cH, \, |h|\ge th(F)}|h|^2=o(n^2)$.

\bigskip

\textbf{Funding}: Research supported by the National Research, Development and Innovation Office - NKFIH under the grants KH 130371, SNN 129364, FK 132060, and KKP-133819.


\begin{thebibliography}{99}


\bibitem{as} N. Alon, C. Shikhelman. Many $T$ copies in $H$-free graphs. \textit{Journal of Combinatorial
Theory, Series B} \textbf{121} (2016) 146--172.

\bibitem{sgmp} S. English, D. Gerbner, A. Methuku, C. Palmer.  On the weight of Berge-$F$-free hypergraphs. {\it Electronic
Journal of Combinatorics} {\bf 26} (2019) P4.7.


\bibitem{erd1} P. Erd\H os. Some recent results on extremal problems in graph theory, \textit{Theory
of Graphs} (Internl. Symp. Rome), 118--123, 1966.

\bibitem{erd2} P. Erd\H os. On some new inequalities concerning extremal properties of graphs, in Theory of Graphs  (ed P.
Erd\H os, G. Katona), Academic Press, New York, 77--81, 1968.

\bibitem{efr} P. Erd\H os, P. Frankl, and V. R\"odl, The asymptotic number of graphs not containing a
fixed subgraph and a problem for hypergraphs having no exponent. \textit{Graphs Combin.}
\textbf{2}(2), 113--121, 1986.

\bibitem{ES1966} P. Erd\H os, M. Simonovits. A limit theorem in graph theory. \textit{Studia Sci.
Math. Hungar.} \textbf{1}, 51--57, 1966.

\bibitem{ES1946} P. Erd\H os, A. H. Stone. On the structure of linear graphs. \textit{Bulletin of the
American Mathematical Society} \textbf{52}, 1087--1091, 1946.

\bibitem{fkl} Z. F\"uredi, A. Kostochka, R. Luo. Avoiding long Berge cycles. \textit{Journal of Combinatorial Theory, Series B} {\bf 137} (2019) 55--64.


\bibitem{gerb} D. Gerbner, A note on the uniformity threshold for Berge hypergraphs, arXiv:2111.00356


\bibitem{gmp} D. Gerbner, A. Methuku, C. Palmer. General lemmas for Berge-Tur\'an hypergraph problems. {\it European Journal of Combinatorics} {\bf 86} (2020) Article 103082.

\bibitem{gmv} D. Gerbner, A. Methuku, M. Vizer. Asymptotics for the Tur\'an number of Berge-$K_{2,t}$. \textit{Journal of Combinatorial Theory, Series B} {\bf 137} (2019) 264--290.

\bibitem{gp1} D. Gerbner, C. Palmer, Extremal Results for Berge hypergraphs, \textit{SIAM Journal on Discrete Mathematics}, \textbf{31} (2017)
2314--2327.

\bibitem{gp2} D. Gerbner, C. Palmer. Counting copies of a fixed subgraph in $ F $-free graphs. {\it European Journal of Combinatorics} {\bf 82} (2019) Article 103001. 

\bibitem{gp} D. Gerbner, B. Patk\'os. Extremal Finite Set Theory,
1st Edition, CRC Press, 2018


\bibitem{ggnpxz} Debarun Ghosh, Ervin Győri, Judit Nagy-György, Addisu Paulos,
Chuanqi Xiao, Oscar Zamora, Book free 3-Uniform Hypergraphs, arXiv:2110.01184

\bibitem{GMT}
D. Gr\'osz, A. Methuku, C. Tompkins. Uniformity thresholds for the asymptotic size of extremal Berge-$F$-free hypergraphs. {\it European Journal of Combinatorics} {\bf 88} (2020) Article 103109.

\bibitem{gyori} E. Győri. Triangle-free hypergraphs. Comb. Probab. Comput., 15(1-2):185--191, 2006.

\bibitem{gyolem} E. Gy\H ori, N. Lemons. Hypergraphs with no cycle of a given length. \textit{Combinatorics,
Probability and Computing} {\bf 21} (2012) 193--201.

\bibitem{lw} L. Lu, Z. Wang, On the cover Turán number of Berge hypergraphs. \textit{European Journal of Combinatorics}
\textbf{98}, Article 103416, 2021.

\bibitem{sim2} M. Simonovits. A method for solving extremal problems in graph theory, stability
problems, in: Theory of Graphs, Proc. Colloq., Tihany, 1966, Academic Press, New
York, (1968), pp. 279–319.

\bibitem{sim} M. Simonovits. Extremal graph problems with symmetrical extremal graphs. Additional chromatic conditions, \textit{Discrete Math.} \textbf{7} 349--376, 1974.

\end{thebibliography}
\end{document}